\documentclass[11pt,reqno]{amsart}
\usepackage{indentfirst, amssymb,amsmath,amsthm, mathrsfs,cite,graphicx,float}
\newtheorem*{theoA}{Theorem A}
\newtheorem*{theoB}{Theorem B}
\newtheorem*{theoC}{Theorem C}
\newtheorem*{theoD}{Theorem D}
\newtheorem*{theoE}{Theorem E}
\newtheorem*{theoF}{Theorem F}
\newtheorem*{theoG}{Theorem G}

\newtheorem{theo}{Theorem}[section]
\newtheorem{lem}{Lemma}[section]

\newtheorem{cor}{Corollary}[section]

\newtheorem{que}{Question}[section]

\newtheorem{open problem}{Open problem}[section]
\newcommand{\pa}{\partial}
\newcommand{\ol}{\overline}
\newcommand{\D}{\mathbb{D}}
\newcommand{\C}{\mathbb{C}}
\newcommand{\N}{\mathbb{N}}
\newcommand{\be}{\begin{equation}}
\newcommand{\ee}{\end{equation}}
\newcommand{\bs}{\begin{small}}
\newcommand{\es}{\end{small}}
\newcommand{\beas}{\begin{eqnarray*}}
\newcommand{\eeas}{\end{eqnarray*}}
\newcommand{\bea}{\begin{eqnarray}}
\newcommand{\eea}{\end{eqnarray}}
\renewcommand{\epsilon}{\varepsilon}
\numberwithin{equation}{section}

\begin{document}
\title[The Bohr's Phenomenon]{The Bohr's Phenomenon for the class of K-quasiconformal harmonic mappings}
\author[R. Biswas and R. Mandal]{Raju Biswas and Rajib Mandal}
\date{}
\address{Department of Mathematics, Raiganj University, Raiganj, West Bengal-733134, India.}
\email{rajubiswasjanu02@gmail.com}
\address{Department of Mathematics, Raiganj University, Raiganj, West Bengal-733134, India.}
\email{rajibmathresearch@gmail.com}
\maketitle
\let\thefootnote\relax
\footnotetext{Type set by \AmS -\LaTeX}
\begin{abstract} 
The primary objective of this paper is to establish several sharp versions of improved Bohr inequality, refined Bohr-type inequality, and refined Bohr-Rogosinski inequality for the class 
of $K$-quasiconformal sense-preserving harmonic mappings $f=h+\ol{g}$ in the unit disk $\D := \{z\in\C : |z| < 1\}$. In order to achieve these objectives, we employ the non-
negative quantity $S_\rho(h)$ and the concept of replacing the initial coefficients of the majorant series by the absolute values of the analytic function and its derivative, as well as other various settings.
Moreover, we obtain the sharp Bohr-Rogosinski radius for harmonic mappings in the unit disk by replacing the bounding condition on the analytic function $h$ with the half-plane condition.
\end{abstract}
\section{Introduction and Preliminaries}
Let $f$ be a bounded analytic function on the open unit disk $\mathbb{D}$ with the Taylor series expansion  
\bea\label{e1}f(z)=\sum_{n=0}^{\infty} a_nz^n.\eea 
Then, 
\bea\label{e2} \sum_{n=0}^{\infty}|a_n|r^n\leq \Vert f\Vert_\infty\quad\text{for}\quad|z|=\rho\leq\frac{1}{3},\eea
where $\Vert f\Vert_\infty:=\sup_{z\in\mathbb{D}} |f(z)|$. It is observed that, if $|f(z)|\leq 1$ in $\mathbb{D}$ and $|f(z_0)|=1$ for some point $z_0\in\mathbb{D}$, then $f(z)$ reduces to a unimodular constant function (see \cite[Strict Maximum Principle (Complex Version), P. 88]{N1}).
In this context, the quantity $1/3$ is known as Bohr radius and it can't be improved. The inequality (\ref{e2}) is known as the Bohr inequality. 
In fact, H. Bohr \cite{5} derived the inequality (\ref{e2}) for values of $r\leq 1/6$. However, subsequently Weiner, Riesz, and Schur \cite{200} independently improved it to $1/3$.\\[2mm]
\indent  Analytic functions of the form (\ref{e1}) with modulus less than 1 satisfying (\ref{e2}) for $|z|\leq 1/3$, are said to satisfy the classical Bohr phenomenon. The concept of Bohr phenomenon can be generalized to the class $\mathcal{F}$ 
consisting of analytic functions of the form (\ref{e1}) which map from $\mathbb{D}$ into a given domain $G\subseteq\mathbb{C}$ such that $f(\mathbb{D})\subseteq G$.
The class $\mathcal{F}$ is said to satisfy the Bohr phenomenon if there exists largest radius $\rho_{\Theta}\in(0, 1)$ such that (\ref{e2}) holds for $|z|=\rho \leq\rho_{\Theta}$.
Here $\rho_{\Theta}$ is known as Bohr radius for the class $\mathcal{F}$.
We refer to \cite{1,2,303,301,210,211,212,304,3,6,7,8,9,9a,10,11,306,307,300,13,14,15,309,20,21,R1,24} and the references listed therein for an in-depth investigation on several 
other aspects of Bohr's inequality. 
In their study, Boas and Khavinson \cite{4} have extended the notion of the Bohr radius to encompass the case of several complex variables. They have furthermore identified the 
multidimensional Bohr radius as a significant contribution to this field of research. 
A considerable number of researchers have built upon this foundation, extending and generalizing the phenomenon across diverse contexts (see \cite{213,214,215}).
Another concept that has been widely discussed is the Hankel determinant of the logarithmic coefficients of univalent functions. See \cite{V1,103} and the references therein for 
some recent results on this topic.\\[2mm]
\indent 
In addition to the notion of the Bohr radius, there is another concept known as the Rogosinski radius \cite{24a}, which is defined as follows:
Let $f(z)=\sum_{n=0}^{\infty}a_nz^n$ be analytic in $\mathbb{D}$ such that $|f(z)|<1$ in $\Bbb{D}$. Then, for every $N\geq 1$, we have 
$\left|S_N(z)\right|=\left|\sum_{n=0}^{N-1}a_nz^n\right|<1$ in the disk $|z|<1/2$, where $S_N(z)$ denotes partial sum of $f$. The number $1/2$ is the best possible. Motivated 
by the Rogosinski radius, Kayumov and Ponnusamy\cite{300} have considered the Bohr-Rogosinski sum $R_N^f(z)$ which is defined as
\beas R_N^f(z):=|f(z)|+\sum_{n=N}^{\infty}|a_n||z|^n, \quad\text{where}\quad N\in\N.\eeas
It is evident that $|S_N(z)|=\left|f(z)-\sum_{n=N}^{\infty}a_nz^n\right|\leq R_N^f(z)$. Moreover, the Bohr-Rogosinski sum
$R_N^f(z)$ is related to the classical Bohr sum (Majorant series) in which $N=1$ and $f(z)$ is replaced by $f(0)$. Let $f$ be an analytic function in $\D$ with $|f(z)|<1$
in $\D$. Kayumov and Ponnusamy \cite{300} defined the Bohr-Rogosinski radius as the largest number $\rho_0\in(0, 1)$ such that the inequality $R_N^f(z)\leq 1$ holds for $|z|<\rho_0$.\\[2mm] 
\indent Kayumov and Ponnusamy \cite{300} obtained the following results pertaining to the Bohr-Rogosinski radius.
\begin{theoA}\cite{300}
Let $f(z)=\sum_{n=0}^{\infty} a_nz^n$ be analytic in $\mathbb{D}$ and $|f(z)|\leq 1$. Then
\beas |f(z)|+\sum_{n=N}^{\infty}|a_n||z|^n\leq 1\eeas
for $|z|=\rho\leq R_N$, where $R_N$ is the positive root of the equation $\psi_N(\rho)=0$, $\psi_N(\rho)=2(1+\rho)\rho^N-(1-\rho)^2$. The radius $R_N$ is the best possible. Moreover, 
\beas|f(z)|^2+\sum_{n=N}^{\infty}|a_n||z|^n\leq 1\eeas
for $|z|=\rho\leq R_N'$, where $R_N'$ is the positive root of the equation $(1+\rho)\rho^N-(1-\rho)^2=0$. The radius $R_N'$ is the best possible.
\end{theoA}
Before proceeding with the discussion, and in order to contextualize the recent results, it is essential to introduce the requisite notations.
\noindent Let $h$ be an analytic function in $\mathbb{D}$ and $\mathbb{D}_\rho:=\{z\in\mathbb{C}: |z|<\rho \}$ for $0<\rho<1$. Let $S_\rho(h)$ denotes the planar integral  
\beas S_\rho(h)=\int_{\mathbb{D}_\rho} |h'(z)|^2 dA(z).\eeas
If $h(z)=\sum_{n=0}^\infty a_n z^n$, then it is well known that $S_\rho(h)/\pi=\sum_{n=1}^\infty n|a_n|^2 \rho^{2n}$ and if $h$ is univalent, then $S_\rho(h)$ is the area of the image $h(\mathbb{D}_\rho)$. \\[1mm]
\indent In addition, if $f(z)$ and $g(z)$ are analytic in $\D$. We say that $f$ is quasi-subordinate to $g$ relative to $\phi$, denoted by $f(z)\prec_qg(z)$ in $\D$ relative to $\phi(z)$, if there exist two functions $\phi$ and $\omega$, analytic in $\D$, satisfying $\omega(0)=0$, $|\omega(z)|\leq 1$, and $|\phi(z)|\leq 1$ for $|z|<1$ such that $f(z)=\phi(z)g(\omega(z))$.\\[2mm]
In 2018, Kayumov and Ponnusamy \cite{14} obtained the following improved versions of Bohr's inequality for the bounded analytic functions in $\mathbb{D}$.
\begin{theoB}\cite{14}
Let $f(z)=\sum_{n=0}^{\infty} a_nz^n$ be analytic in $\mathbb{D}$, $|f(z)|\leq 1$ and $S_\rho$ denotes the area of the image of the subdisk $|z|<\rho$ under mapping $f$. Then
\beas \sum_{n=0}^{\infty} |a_n|\rho^n+\frac{16}{9}\left(\frac{S_\rho(f)}{\pi}\right)\leq 1\quad\text{for}\quad \rho\leq\frac{1}{3}.\eeas
The numbers $1/3$, $16/9$ cannot be improved. Moreover, 
\beas |a_0|^2+\sum_{n=1}^{\infty} |a_n| \rho^n+\frac{9}{8}\left(\frac{S_\rho(f)}{\pi}\right)\leq 1\quad\text{for}\quad \rho\leq\frac{1}{2}.\eeas
The numbers $1/2$, $9/8$ cannot be improved. \end{theoB}
\noindent Based on the initiation of Kayumov and Ponnusamy \cite{300,14}, Liu {\it et al.} \cite{215a} explored several forms of Bohr-type inequalities and we recall one of them here.
\begin{theoC}\cite{215a}
Let $f(z)=\sum_{n=0}^\infty a_nz^n$ be analytic in $\mathbb{D}$ and $|f(z)|<1$ in $\mathbb{D}$. Then,
\beas |f(z)|+|f'(z)|\rho+\sum_{n=2}^\infty |a_n|\rho^n\leq 1\quad\text{for}\quad \rho\leq\frac{\sqrt{17}-3}{4}.\eeas
The number $(\sqrt{17}-3)/4$ is the best possible. \end{theoC}
Moreover, a number of authors have investigated additional extensions of this kind (see \cite{2a,300a,18a}). In 2020, Ponnusamy {\it et al.} \cite{23a} obtained the following refined Bohr inequality by applying a refined version of the coefficient inequalities.
\begin{theoD}\cite{23a}  Let $f(z)=\sum_{n=0}^{\infty} a_nz^n$ be analytic in $\mathbb{D}$ and $|f(z)|\leq 1$. Then,
\beas \sum_{n=0}^\infty |a_n|\rho^n+\left(\frac{1}{1+|a_0|}+\frac{\rho}{1-\rho}\right)\sum_{n=1}^\infty |a_n|^2 \rho^{2n}\leq 1\eeas
for $\rho\leq 1/(2+|a_0|)$, and the numbers $1/(1+|a_0|)$ and $1/(2+|a_0|)$ cannot be improved. Moreover,
\beas |a_0|^2+\sum_{n=1}^\infty |a_n|\rho^n+\left(\frac{1}{1+|a_0|}+\frac{\rho}{1-\rho}\right)\sum_{n=1}^\infty |a_n|^2 \rho^{2n}\leq 1\eeas
for $\rho\leq 1/2$. The numbers $1/(1+|a_0|)$ and $1/2$ cannot be improved.\end{theoD}
Let $f = u + iv$ be a complex-valued function in a simply connected domain $\Omega$. If $f$ satisfies the Laplace equation $\Delta f =4f_{z\ol z} = 0$, then $f$ is said to be harmonic in $\Omega$. In other words, the functions $u$ and $v$ are real harmonic in $\Omega$. It should be noted that every harmonic mapping $f$ has the canonical representation $f = h + \ol g$, where $h$ and 
$g$ are analytic in $\Omega$, known respectively as the analytic and co-analytic parts of $f$, and $\ol{g(z)}$ denotes the complex conjugate of $g(z)$. This representation is 
unique up to an additive constant (see \cite{20P}). The inverse function theorem and a result of Lewy \cite{18} shows that a harmonic function $f$ is locally univalent in $\Omega$ if, and only if, the Jacobian  of $f$, defined by $J_f(z):=|h'(z)|^2-|g'(z)|^2$ is non-zero in $\Omega$. A harmonic mapping $f$ is locally univalent and sense-preserving
in $\Omega$ if, and only if, $J_f (z) > 0$ in $\Omega$ or equivalently if $h'\not=0$ in $\Omega$
and the dilatation $\omega_f:= \omega=g'/h'$ of $f$ has the property that $|\omega_f| < 1$ in $\Omega$ (see \cite{18}).\\[2mm]
\indent If a locally univalent and sense-preserving harmonic mapping $f = h+\ol{g}$ satisfies the condition $|g'(z)/h'(z)| \leq k < 1$ for $z\in\mathbb{D}$, then $f$ is said to be 
$K$-quasiconformal harmonic mapping on $\mathbb{D}$, where $K = (1+k)/(1-k) \geq 1$ (see \cite{12,22}). Clearly, $k\to 1$ corresponds to the limiting case $K \to\infty$.\\[2mm] 
\indent In 2018, Kayumov {\it et al.} \cite{15a} established the harmonic extension of the classical Bohr theorem and obtained the following results. 
\begin{theoE}\cite{15a} Suppose that $f(z)=h(z)+\ol{g(z)}=\sum_{n=0}^\infty a_n z^n+\ol{\sum_{n=1}^\infty b_n z^n}$ is a sense-preserving $K$-quasiconformal harmonic mapping in $\mathbb{D}$, where $h(z)$ is bounded in $\mathbb{D}$. Then
\beas\sum_{n=0}^\infty |a_n|\rho^n+\sum_{n=1}^\infty |b_n|\rho^n\leq \Vert h(z)\Vert_\infty\quad\text{for}\quad \rho\leq \frac{K+1}{5K+1}.\eeas
The number $(K+1)/(5K+1)$ is sharp. Moreover,
\beas |a_0|^2+\sum_{n=1}^\infty (|a_n|+|b_n|)\rho^n\leq \Vert h(z)\Vert_\infty\quad\text{for}\quad \rho\leq \frac{K+1}{3K+1}.\eeas
The number $(K+1)/(3K+1)$ is sharp.
\end{theoE}
\begin{theoF}\cite{15a} Suppose that $f(z)=h(z)+\ol{g(z)}=\sum_{n=0}^\infty a_n z^n+\ol{\sum_{n=2}^\infty b_n z^n}$ is a sense-preserving $K$-quasiconformal harmonic mapping in $\mathbb{D}$, where $h(z)$ is bounded in $\mathbb{D}$. Then
\beas\sum_{n=0}^\infty |a_n|\rho^n+\sum_{n=2}^\infty |b_n|\rho^n\leq \Vert h(z)\Vert_\infty\quad\text{for}\quad \rho\leq \rho_K,\eeas
where $\rho_K$ is the positive root of the equation 
\beas \frac{\rho}{1-\rho}+\left(\frac{K-1}{K+1}\right) \rho^2\sqrt{\frac{1+\rho^2}{(1-\rho^2)^3}}\sqrt{\frac{\pi^2}{6}-1}=\frac{1}{2}.\eeas
The number $\rho_K$ cannot be replaced by a number greater than $R=R(K)$, where $R$ is the positive root of the equation 
\beas \frac{4R}{1-R}\left(\frac{K}{K+1}\right)+2\left(\frac{K-1}{K+1}\right)\log(1-R)=1.\eeas\end{theoF}
In 2018, Kayumov {\it et al.} \cite{15a} established the following result by imposing the half-plane condition in place of the bounding condition on $h$.
\begin{theoG}
Suppose that $f(z)=h(z)+\ol{g(z)}=\sum_{n=0}^\infty a_n z^n+\ol{\sum_{n=1}^\infty b_n z^n}$ is a sense-preserving $K$-quasiconformal harmonic mapping in $\mathbb{D}$, where $h(z)$ satisfies the conditions $\text{Re}\;h(z)\leq 1$ in $\D$ and $h(0)=a_0>0$. Then
\beas a_0+\sum_{n=1}^\infty |a_n|\rho^n+\sum_{n=1}^\infty |b_n|\rho^n\leq1\quad\text{for}\quad \rho\leq \frac{K+1}{5K+1}.\eeas
The number $(K+1)/(5K+1)$ is sharp.
\end{theoG}
\noindent In light of the aforementioned findings, several questions naturally arise with regard to this study.
\begin{que}\label{Q1}
Can we establish a sharply improved Bohr inequality and a sharply refined Bohr-type inequality using non-negative $S_\rho(h)$ without compromising the radius?
\end{que}
\begin{que}\label{Q2} Can we establish the refined Bohr-Rogosinski inequality of \textrm{Theorem E}? \end{que}
\begin{que}\label{Q3} Can we establish several sharply improved versions of the Bohr inequality of \textrm{Theorem E} by the idea of replacing the initial coefficients of the majorant series with the absolute values of the analytic function and its derivative?   \end{que} 
\begin{que}\label{Q4} Can we establish the sharply refined version of the Bohr-type inequality in the context of \textrm{Theorem F} by replacing the initial coefficients of the majorant series by the absolute values of the analytic function and its derivative?  \end{que}
\begin{que}\label{Q5} Can we establish the sharp version of \textrm{Theorem F}? \end{que}
\begin{que}\label{Q6} Can we establish the sharp Bohr-Rogosinski inequality in the context of \textrm{Theorem G}? \end{que}
The purpose of this paper is primarily to provide the affirmative answers to Questions \ref{Q1}-\ref{Q6}.
\section{Some lemmas}
The following lemmas are needed for this paper and will be used to prove the main results.
\begin{lem}\label{lem1} \cite[Pick's invariant form of Schwarz’s lemma]{201} Suppose $f$ is analytic in $\mathbb{D}$ with $|f(z)|\leq1$, then 
\beas |f(z)|\leq \frac{|f(0)|+|z|}{1+|f(0)||z|}\quad\text{for}\quad z\in\mathbb{D}.\eeas\end{lem}
\begin{lem}\cite{7a,25}\label{lem2} Suppose $f$ is analytic in $\mathbb{D}$ with $|f(z)|\leq1$, then we have 
\beas \frac{\left|f^{(n)}(z)\right|}{n!}\leq \frac{1-|f(z)|^2}{(1-|z|)^{n-1}(1-|z|^2)}\quad\text{for each}\quad n\geq 1\quad \text{and}\quad |z|<1.\eeas\end{lem}
\begin{lem}\cite{15a}\label{lem3} Suppose that $h(z)=\sum_{n=0}^\infty a_nz^n$ and $g(z)=\sum_{n=0}^\infty b_nz^n$ are two analytic functions in $\mathbb{D}$ such that $|g’(z)|\leq k|h’(z)|$ in $\mathbb{D}$ and for some $k\in [0,1)$ with $|h(z)|\leq 1$. Then,
\beas\sum_{n=1}^\infty |b_n|^2\rho^n\leq k^2 \sum_{n=1}^\infty |a_n|^2\rho^n\;\;\text{for}\;\;|z|=\rho<1.\eeas
\end{lem}
By employing the concept of quasi-subordination and the result derived in \cite{2}, Liu {\it et al.} \cite{20} have established the following result.
\begin{lem}\cite[Proof of Theorem 2]{20}\label{lem30} Suppose that $h(z)=\sum_{n=0}^\infty a_nz^n$ and $g(z)=\sum_{n=0}^\infty b_nz^n$ are two analytic functions in $\mathbb{D}$ such that $|g’(z)|\leq k|zh’(z)|$ in $\mathbb{D}$ for $k\in [0,1)$. Then,
\beas\sum_{n=1}^\infty n |b_n|\rho^{n-1}\leq k \sum_{n=1}^\infty n|a_n|\rho^n\quad\text{for}\quad |z|=\rho\leq 1/3.\eeas
\end{lem}
\begin{lem}\cite{18a}\label{lem5} Suppose $f$ is analytic in $\mathbb{D}$ with $|f(z)|\leq1$, then for any $N\in\mathbb{N}$, the following inequality holds:
\beas\sum_{n=N}^\infty |a_n|\rho^n+\text{sgn}(t)\sum_{n=1}^t |a_n|^2 \frac{\rho^N}{1-\rho}+\left(\frac{1}{1+a_0}+\frac{\rho}{1-\rho}\right)\sum_{n=t+1}^\infty |a_n|^2 \rho^{2n}\leq \frac{(1-|a_0|^2)\rho^N}{1-\rho}\eeas
for $\rho\in[0,1)$, where $t=\lfloor (N-1)/2\rfloor$ and $\lfloor x\rfloor$ denotes the largest integer not exceeding the real number $x$.
\end{lem}
\begin{lem}\cite{1}\label{lem6} If $p(z)=\sum_{k=0}^\infty p_k z^k$ is analytic in $\D$ such that $\text{Re}\;p(z)>0$ in $\D$, then $|p_k|\leq 2\;\text{Re}\;p_0$ for all $k\geq 1$.\end{lem}
\section{Main results}
In the following, we obtain the sharp improved version of Bohr inequality in the settings of \textrm{Theorem E} by utilizing the non-negative quantity 
$S_\rho(h)$ without reducing the radius. 
\begin{theo}\label{T1} Suppose that $f(z)=h(z)+\ol{g(z)}=\sum_{n=0}^\infty a_n z^n+\ol{\sum_{n=1}^\infty b_n z^n}$ is a sense-preserving $K$-quasiconformal harmonic mapping in $\mathbb{D}$, where $h(z)$ is bounded in $\mathbb{D}$. Then
\beas\sum_{n=0}^\infty |a_n|\rho^n+\sum_{n=1}^\infty |b_n|\rho^n+\frac{8K^2(3K+1)^2}{(5K+1)^2(K+1)^2}\frac{S_\rho(h)}{\pi}\leq \Vert h(z)\Vert_\infty\quad\text{for}\quad \rho\leq \rho_0=\frac{K+1}{5K+1}.\eeas
The numbers $8K^2(3K+1)^2/((5K+1)^2(K+1)^2)$ and  $(K+1)/(5K+1)$ cannot be replaced by larger values.
\end{theo}
\begin{proof} For simplicity, we suppose that $\Vert h(z)\Vert_\infty\leq 1$. In view of \textrm{lemma \ref{lem2}}, we have $|a_n|\leq 1-|a_0|^2$ for $n\geq 1$. Since $f$ is locally univalent and $K$-quasiconformal sense-preserving harmonic mapping on $\mathbb{D}$, Schwarz's lemma gives that the dilatation $\omega=g'/h'$ is analytic in $\mathbb{D}$ and $|\omega(z)|\leq k$, {\it i.e.}, $|g'(z)|\leq k|h'(z)|$ in $\mathbb{D}$, where $K = (1+k)/(1-k) \geq 1$, $k\in[0,1)$. Let $|a_0|=a\in[0,1)$.
By \textrm{Lemma \ref{lem3}}, we have 
\bea\label{E1} \sum_{n=1}^\infty |b_n|^2\rho^n\leq k^2 \sum_{n=1}^\infty |a_n|^2\rho^n\leq k^2(1-a^2)^2\frac{\rho}{1-\rho}.\eea
Using (\ref{E1}) and in view of Cauchy-Schwarz inequality, we have 
\beas \sum_{n=1}^\infty |b_n|\rho^n\leq \left(\sum_{n=1}^\infty |b_n|^2\rho^n\right)^{1/2}\left(\sum_{n=1}^\infty\rho^n\right)^{1/2} \leq k(1-a^2)\frac{\rho}{1-\rho}.\eeas 
From the definition of $S_\rho(h)$, we have
\bea\label{E2} \frac{S_\rho(h)}{\pi}=\sum_{n=1}^\infty n|a_n|^2 \rho^{2n}\leq (1-a^2)^2\sum_{n=1}^\infty n\rho^{2n}=(1-a^2)^2\frac{\rho^2}{(1-\rho^2)^2}.\eea
Therefore,
\beas \sum_{n=0}^\infty |a_n|\rho^n+\sum_{n=1}^\infty |b_n|\rho^n+\lambda\frac{S_\rho(h)}{\pi}&\leq& a+(1+k)(1-a^2)\frac{\rho}{1-\rho}+\lambda(1-a^2)^2\frac{\rho^2}{(1-\rho^2)^2}\\
&=&1+F_1(a,\rho),\eeas
where 
\beas F_1(a,\rho)&=&(1+k)(1-a^2)\frac{\rho}{1-\rho}+(1-a^2)^2\frac{\lambda\rho^2}{(1-\rho^2)^2}-(1-a)\\
&=&\frac{(1-a^2)}{2}\left(1+\left(\frac{2(1+k)\rho}{1-\rho}-1\right)+(1-a^2)\frac{2\lambda\rho^2}{(1-\rho^2)^2}-\frac{2}{1+a}\right).\eeas
Differentiating partially with respect to $\rho$, we get
\beas \frac{\pa}{\pa \rho} F_1(a,\rho)=\frac{(1+k)(1-a^2)}{(1-\rho)^2}+(1-a^2)^2\frac{2\lambda\rho(1+\rho^2)}{(1-\rho^2)^3}>0.\eeas 
Therefore $F_1(a,\rho)$ is a monotonically increasing function of $\rho$ in $[0,1)$ and it follows that $F_1(a,\rho)\leq F_1(a,\rho_0)$ for $\rho\leq \rho_0=1/(2k+3)$. Now,
\beas F_1(a,\rho_0)=\frac{(1-a^2)}{2}\left(1+(1-a^2)\frac{\lambda(2k+3)^2}{8(k+1)^2(k+2)^2}-\frac{2}{1+a}\right).\eeas
Let $F_2(a)=1+\lambda(2k+3)^2(1-a^2)/(8(k+1)^2(k+2)^2)-2/(1+a)$, $a\in[0,1)$. It is easy to see that
\beas F_2(0)=\frac{\lambda(2k+3)^2}{8(k+1)^2(k+2)^2}-1\quad\text{and}\quad \lim_{a\to1^-}F_2(a)=0.\eeas
Differentiating $F_2(a)$ with respect to $a$, we have
\beas F_2'(a)&=&\frac{-2a\lambda(2k+3)^2}{8(k+1)^2(k+2)^2}+\frac{2}{(1+a)^2}\\
&=&\frac{2}{(1+a)^2}\left(1-\frac{\lambda(2k+3)^2}{8(k+1)^2(k+2)^2}a(1+a)^2\right)\\
&\geq&\frac{2}{(1+a)^2}\left(1-\frac{\lambda(2k+3)^2}{2(k+1)^2(k+2)^2}\right)\geq 0,\eeas 
if $\lambda\leq 2(k+1)^2(k+2)^2/(2k+3)^2=8K^2(3K+1)^2/((5K+1)^2(K+1)^2)$.
Therefore, $F_2(a)$ is a monotonically increasing function of $a$ in $[0,1)$ and it follows that $F_2(a)\leq 0$ for $a\in[0,1)$ and $\lambda\leq 8K^2(3K+1)^2/((5K+1)^2(K+1)^2)$. Therefore, we have
\beas \sum_{n=0}^\infty |a_n|\rho^n+\sum_{n=1}^\infty |b_n|\rho^n+\frac{8K^2(3K+1)^2}{(5K+1)^2(K+1)^2}\frac{S_\rho(h)}{\pi}\leq 1\eeas
for $\rho\leq \rho_0=1/(2k+3)=(K+1)/(5K+1)$.\\[2mm]
\indent To prove the sharpness of the result, we consider the function $f_1(z)=h_1(z)+\ol{g_1(z)}$ in $\mathbb{D}$ such that 
\beas h_1(z)=\frac{a-z}{1-az}=A_0+\sum_{n=1}^\infty A_n z^n,\eeas
where $A_0=a$, $A_n=-(1-a^2)a^{n-1}$ for $n\geq 1$, $a\in[0,1)$ and $g_1(z)=\lambda k \sum_{n=1}^\infty A_n z^n$, where $|\lambda|=1$ and $k=(K-1)/(K+1)$.
Thus,
\beas S_1:&=&\sum_{n=0}^\infty |A_n|\rho^n+\sum_{n=1}^\infty |k\lambda A_n|\rho^n+\frac{8K^2(3K+1)^2}{(5K+1)^2(K+1)^2}\frac{S_\rho(h_1)}{\pi}\\
&=&a+\frac{(1+k)(1-a^2)}{a}\sum_{n=1}^\infty (a\rho)^{n}+\frac{8K^2(3K+1)^2}{(5K+1)^2(K+1)^2}\sum_{n=1}^\infty n|A_n|^2 \rho^{2n}\\
&=& 1+(1-a) F_3(a,\rho),\eeas
where 
\beas F_3(a,\rho)=\frac{2K(1+a)\rho}{(K+1)(1-a\rho)}+\frac{8K^2(3K+1)^2}{(5K+1)^2(K+1)^2}\frac{(1-a^2)(1+a)\rho^2}{(1-a^2\rho^2)^2}-1.\eeas
Differentiating partially $F_3(a,\rho)$ with respect to $\rho$, we have 
\beas \frac{\pa}{\pa \rho}F_3(a,\rho)=\frac{2K(1+a)}{(K+1)(1-a\rho)^2}+\frac{16K^2(3K+1)^2}{(5K+1)^2(K+1)^2}\frac{(1-a^2)(1+a)\rho(1+\rho^2)}{(1-a^2\rho^2)^3}>0\eeas
for $\rho\in(0,1)$.
Therefore, $F_3(a,\rho)$ is a strictly increasing function of $\rho\in(0,1)$. Thus, for $\rho>(K+1)/(5K+1)$, we have
\beas F_3(a,\rho)&>&F_3(a,(K+1)/(5K+1))\\&=&\frac{2K(1+a)}{4K}+\frac{8K^2(3K+1)^2(1-a^2)(1+a)}{((5+a)K+1+a)^2((5-a)K+1-a)^2}-1\to 0\;\text{as}\;a\to1^-.\eeas 
 Hence, $S_1:=1+(1-a) F_3(a,\rho)>1$ for $\rho>(K+1)/(5K+1)$. 												
This shows that $(K+1)/(5K+1)$ is the best possible. This completes the proof.
\end{proof}
The following two results are the sharp improved versions of Bohr inequality in the settings of \textrm{Theorem E} by the concept of replacing $|a_0|$ with $|h(z)|^s$ and $|a_1|$ with $|h'(z)|$ in the majorant series, where $s=1,2$.
\begin{theo}\label{TT1} Suppose that $f(z)=h(z)+\ol{g(z)}=\sum_{n=0}^\infty a_n z^n+\ol{\sum_{n=1}^\infty b_n z^n}$ is a sense-preserving $K$-quasiconformal harmonic mapping in $\mathbb{D}$, where $\Vert h(z)\Vert_\infty\leq 1$ in $\mathbb{D}$. Then
\beas |h(z)|+|h'(z)|\rho+\sum_{n=2}^\infty |a_n|\rho^n+\sum_{n=1}^\infty |b_n|\rho^n\leq 1\quad\text{for}\quad \rho\leq \rho_0\leq\sqrt{2}-1 ,\eeas
where $\rho_0\in(0,\sqrt{2}-1)$ is the unique root of the equation
\beas (1-\rho)\left(\rho^2+2\rho-1\right)+2\rho(1+\rho)^2\left(\frac{K-1}{K+1}+\rho\right)=0.\eeas
The number $\rho_0$ is sharp.
\end{theo}
\begin{proof}
Let $F(x)=x+\alpha(1-x^2)$, where $0\leq x\leq x_0(\leq1)$ and $\alpha\geq 0$. Then, $F'(x)=1-2\alpha x$ and $F''(x)=-2\alpha\leq 0$. Therefore, $F'(x)$ is a monotonically 
decreasing function of $x$ and it follows that $F'(x)\geq F'(1)=1-2\alpha\geq 0$ for $\alpha\leq 1/2$. Hence, we have $F(x)\leq F(x_0)$ for $0\leq \alpha\leq 1/2$.  
By using similar arguments as in the proof of \textrm{Theorem \ref{T1}}, and in view of \textrm{Lemmas \ref{lem1}, \ref{lem2}}, and \ref{lem3} together with the condition $|g'(z)|\leq k|h'(z)|$, we have
\beas&& \sum_{n=1}^\infty |b_n|^2\rho^n\leq k^2 \sum_{n=1}^\infty |a_n|^2\rho^n\leq k^2(1-a^2)^2\frac{\rho}{1-\rho},\\[2mm]
&&\sum_{n=1}^\infty |b_n|\rho^n\leq \left(\sum_{n=1}^\infty |b_n|^2\rho^n\right)^{1/2}\left(\sum_{n=1}^\infty\rho^n\right)^{1/2} \leq k(1-a^2)\frac{\rho}{1-\rho},\\[2mm]
&&|h(z)|\leq \frac{|h(0)|+|z|}{1+|h(0)||z|}\quad\text{and}\quad \left|\frac{h^{(n)}(z)}{n!}\right|\leq \frac{1-|h(z)|^2}{(1+|z|)(1-|z|)^n}\quad\text{for}\quad n\geq 1,\eeas
where $|a_0|=a\in[0,1)$.
It is evident that $\rho/(1-\rho^2)\leq 1/2$ for any $\rho\in[0,\sqrt{2}-1]$.
Therefore,
\beas &&|h(z)|+|h'(z)|\rho+\sum_{n=2}^\infty |a_n|\rho^n+\sum_{n=1}^\infty |b_n|\rho^n\\[2mm]
&&\leq \frac{a+\rho}{1+a\rho}+\frac{\rho}{1-\rho^2}\left(1-\left(\frac{a+\rho}{1+a\rho}\right)^2\right)+(1-a^2)\frac{\rho^2}{1-\rho}+k(1-a^2)\frac{\rho}{1-\rho}\\[2mm]
&&= \frac{a+\rho}{1+a\rho}+\frac{\rho(1-a^2)}{(1+a\rho)^2}+(1-a^2)\frac{\rho}{1-\rho}(k+\rho)\\[2mm]
&&=1+\frac{(1-a) F_4(a,\rho)}{(1+a\rho)^2(1-\rho)},\eeas
where $F_4(a,\rho)=(1+a)\rho(1-\rho)+(1+a)\rho(k+\rho)(1+a\rho)^2-(1-\rho)^2(1+a\rho)$ and the first inequality hold for any $\rho\in[0,\sqrt{2}-1]$. Differentiating partially $F_4(a,\rho)$ twice with respect to $a$, we have 
\beas\frac{\pa}{\pa a}F_4(a,\rho)&=&\rho(1-\rho)+\rho(k+\rho)(1+a\rho)^2+2(1+a)\rho^2(k+\rho)(1+a\rho)-\rho(1-\rho)^2\\[2mm]
\frac{\pa^2}{\pa a^2}F_4(a,\rho)&=&2\rho^2(k+\rho)(1+a\rho)+2\rho^2(k+\rho)(1+a\rho)+2(1+a)\rho^3(k+\rho)\geq 0.\eeas
Therefore, $\frac{\pa}{\pa a}F_4(a,\rho)$ is a monotonically increasing function of $a\in[0,1)$ and hence, we have 
\beas \frac{\pa}{\pa a}F_4(a,\rho)\geq \frac{\pa}{\pa a}F_5(0,\rho)
=\rho^3+2(k+1)\rho^2+k\rho\geq 0.\eeas 
Therefore, $F_4(a,\rho)$ is a monotonically increasing function of $a\in[0,1)$ and it follows that 
\beas F_4(a,\rho)\leq F_4(1,\rho)
=(1-\rho)\left(2\rho+\rho^2-1\right)+2\rho(k+\rho)(1+\rho)^2\leq 0\quad\text{for}\quad \rho\leq \rho_0,\eeas 
where $\rho_0$ is the smallest root of the equation 
\bea\label{p4} F_5(\rho):=(1-\rho)\left(2\rho+\rho^2-1\right)+2\rho(k+\rho)(1+\rho)^2=0,\eea
where $k=(K-1)/(K+1)$. Therefore, $F_5(\sqrt{2}-1)=4 (\sqrt{2}-1) (\sqrt{2}-1+ k)>0$, $F_5(0)=-1<0$ and
\beas F_5'(\rho)=3+2k+2\rho+8k\rho+9\rho^2+6k\rho^2+8\rho^3\geq 0\quad\text{for}\quad \rho\in(0,\sqrt{2}-1),\eeas
which shows that $F_5(\rho)$ is a monotonically increasing function of $\rho$. Therefore, $\rho_0\in(0,\sqrt{2}-1)$ is the unique positive root of the equation (\ref{p4}).
It is evident that $2\rho-(1-\rho^2)>0$ for $\rho>\sqrt{2}-1$ and thus, $F_5(\rho)>0$ for $\rho>\sqrt{2}-1$. Therefore, we must have $\rho_0\leq\sqrt{2}-1$.\\[2mm]
\indent To prove the sharpness of the result, we consider the function $f_2(z)=h_2(z)+\ol{g_2(z)}$ in $\mathbb{D}$ such that 
\beas h_2(z)=\frac{a-z}{1-az}=A_0+\sum_{n=1}^\infty A_n z^n,\eeas
where $A_0=a$, $A_n=-(1-a^2)a^{n-1}$ for $n\geq 1$, $a\in[0,1)$ and $g_2(z)=\lambda k \sum_{n=1}^\infty A_n z^n$, where $|\lambda|=1$ and $k=(K-1)/(K+1)$.
Thus,
\beas&&|h_2(-\rho)|+|h_2'(-\rho)|\rho+\sum_{n=2}^\infty |A_n|\rho^n+\sum_{n=1}^\infty |k\lambda A_n|\rho^n\\[2mm]
&&=\frac{a+\rho}{1+a\rho}+\frac{(1-a^2)\rho}{(1+a\rho)^2}+(1-a^2)\rho\sum_{n=2}^\infty (a\rho)^{n-1}+(1-a^2)k\rho\sum_{n=1}^\infty (a\rho)^{n-1}\\[2mm]
&&=\frac{a+\rho}{1+a\rho}+\frac{(1-a^2)\rho}{(1+a\rho)^2}+\frac{(1-a^2)a\rho^2}{1-a\rho}+\frac{(1-a^2)k\rho}{1-a\rho}\\[2mm]
&&=1+\frac{(1-a)}{(1+a\rho)^2(1-a\rho)}F_6(a,\rho),\eeas
where 
\beas F_6(a,\rho)=(1+a)\rho(1-a\rho)+(1+a)(k+\rho)\rho(1+a\rho)^2-(1-\rho)(1-a\rho)(1+a\rho).\eeas
It is evident that
\beas \lim_{a\to1^-} F_6(a,\rho)
=(1-\rho)\left(2\rho-(1-\rho^2)\right)+2(k+\rho)\rho(1+\rho)^2>0\quad\text{for}\quad \rho>\rho_0, \eeas
where $\rho_0$ is the unique positive root of the equation (\ref{p4}) in $(0,\sqrt{2}-1)$. This shows that the radius $\rho_0$ is the best possible. This completes the proof.
\end{proof}
\begin{theo}\label{T7}  Suppose that $f(z)=h(z)+\ol{g(z)}=\sum_{n=0}^\infty a_n z^n+\ol{\sum_{n=1}^\infty b_n z^n}$ is a sense-preserving $K$-quasiconformal harmonic mapping in $\mathbb{D}$, where $\Vert h(z)\Vert_\infty\leq 1$ in $\mathbb{D}$. Then
\beas |h(z)|^2+\left|h'(z)\right|\rho+\sum_{n=2}^\infty \left|a_n\right|\rho^n+\sum_{n=1}^\infty |b_n|\rho^n\leq 1\quad\text{for}\quad \rho\leq \rho_0\leq \left(\sqrt{5}-1\right)/2,\eeas
where $\rho_0$ is the unique positive root of the equation 
\beas (1-\rho)\left(1-\rho^2-\rho\right)-\rho(1+\rho)^2\left(\rho+\frac{K-1}{K+1}\right)=0.\eeas 
The number $\rho_0$ is sharp.
\end{theo}
\begin{proof}
By using similar arguments as in the proof of \textrm{Theorem \ref{T1}}, and in view of \textrm{Lemmas \ref{lem1}, \ref{lem2}} and \ref{lem3} together with the condition $|g'(z)|\leq k|h'(z)|$, we have
\beas
&&\sum_{n=1}^\infty |b_n|\rho^n\leq \left(\sum_{n=1}^\infty |b_n|^2\rho^n\right)^{1/2}\left(\sum_{n=1}^\infty\rho^n\right)^{1/2} \leq k(1-a^2)\frac{\rho}{1-\rho},\\[2mm]
&&|h(z)|\leq \frac{|h(0)|+|z|}{1+|h(0)||z|}\quad\text{and}\quad \left|\frac{h^{n}(z)}{n!}\right|\leq \frac{1-|h(z)|^2}{(1-|z|)^n(1+|z|)}\;\;\text{for}\;\;n\geq 1,\eeas
where $|a_0|=a\in[0,1)$. 
It is evident that $1-\rho^2-\rho\geq 0$ for $0\leq \rho\leq \left(\sqrt{5}-1\right)/2$.
Thus, we have
\beas |h(z)|^2+|h'(z)|\rho+\sum_{n=2}^\infty \left|a_n\right|\rho^n+\sum_{n=1}^\infty |b_n|\rho^n&\leq& |h(z)|^2+\frac{1-|h(z)|^2}{1-\rho^2}\rho+(1-a^2)\sum_{n=2}^\infty \rho^n\\[2mm]
&&+k(1-a^2)\frac{\rho}{1-\rho}\\[2mm]
&=&\left(1-\frac{\rho}{1-\rho^2}\right) |h(z)|^2+\frac{\rho}{(1-\rho^2)}\\[2mm]
&&+(1-a^2)\frac{\rho^2}{1-\rho}+k(1-a^2)\frac{\rho}{1-\rho}\\[2mm]
&\leq& \frac{1-\rho^2-\rho}{(1-\rho^2)}\left(\frac{a+\rho}{1+a\rho}\right)^2+\frac{\rho}{1-\rho^2}\\[2mm]
&&+(k+1)(1-a^2)\frac{\rho}{1-\rho}-(1-a^2)\rho\\[2mm]
&=&1-\frac{(1-a^2)}{(1-\rho)(1+a\rho)^2}F_7(a,\rho),\eeas
where 
\beas F_7(a,\rho)=(1-\rho^2-\rho)(1-\rho)+(1-\rho)(1+a\rho)^2\rho-(k+1)\rho(1+a\rho)^2\eeas and the second inequality hold for any $\rho\in[0, \left(\sqrt{5}-1\right)/2]$. Differentiating partially $F_7(a,\rho)$ twice with respect to $a$, we have 
\beas\frac{\pa}{\pa a}F_7(a,\rho)=-2(1+a\rho)\rho^2\left(k+\rho\right)\leq 0.\eeas
Therefore, $F_7(a,\rho)$ is a monotonically decreasing function of $a\in[0,1)$ and it follows that 
\beas F_7(a,\rho)\geq F_7(1,\rho)
&=&(1-\rho)(1-\rho^2-\rho)-\rho(1+\rho)^2\left(\rho+k\right)\geq 0\eeas
for $\rho\leq \rho_0$, where $k=(K-1)/(K+1)$ and $\rho_0$ is the smallest positive root of the equation 
\bea\label{p5}F_8(\rho):=(1-\rho)(1-\rho^2-\rho)-\rho(1+\rho)^2\left(\rho+k\right)=0.\eea
Clearly, $1-\rho^2-\rho<0$ for $\rho>\left(\sqrt{5}-1\right)/2$ and thus, we have $F_8(\rho)<0$ for $\rho>\left(\sqrt{5}-1\right)/2$. Hence, we must have $\rho_0\leq \left(\sqrt{5}-1\right)/2$. Also, $F_8(0)=1>0$, $F_8\left(\left(\sqrt{5}-1\right)/2\right)=-(\sqrt{5}-1) (\sqrt{5}+1)^2 (\sqrt{5}-1+ 2 k)/16<0$ and
\beas F_8'(\rho)=-2-2\rho- 3\rho^2-4\rho^3-k(1+4\rho+3\rho^2)\leq 0\quad\text{for}\quad \rho\in\left(0, \left(\sqrt{5}-1\right)/2\right).\eeas 
It's follows that $\rho_0$ is the unique positive root of the equation (\ref{p5}) in $\left(0, \left(\sqrt{5}-1\right)/2\right)$.\\[2mm]
\indent To prove the sharpness of the result, we consider the function $f_3(z)=h_3(z)+\ol{g_3(z)}$ in $\mathbb{D}$ such that 
\beas h_3(z)=\frac{a-z}{1-az}=A_0+\sum_{n=1}^\infty A_n z^n\quad\text{so that}\quad
\frac{h_3^{n}(z)}{n!}=-\frac{a^{n-1}(1-a^2)}{(1-az)^{n+1}}\quad\text{for}\quad n\geq 1,\eeas
where $A_0=a$, $A_n=-(1-a^2)a^{n-1}$ for $n\geq 1$, $a\in[0,1)$ and $g_3(z)=\lambda k \sum_{n=1}^\infty A_n z^n$, where $|\lambda|=1$ and $k=(K-1)/(K+1)$.
Therefore,
\beas&&|h_3(-\rho)|^2+\left|h_3'(-\rho)\right|\rho+\sum_{n=2}^\infty \left|A_n\right|\rho^n+\sum_{n=1}^\infty |k\lambda A_n|\rho^n\\[2mm]
&&=\left(\frac{a+\rho}{1+a\rho}\right)^2+\frac{(1-a^2)\rho}{(1+a\rho)^2}+(1-a^2)\rho\sum_{n=2}^\infty (a\rho)^{n-1}+(1-a^2)k\rho\sum_{n=1}^\infty (a\rho)^{n-1}\\[2mm]
&&=1+\frac{(1-a^2)}{(1+a\rho)^2(1-a\rho)}F_9(a,\rho),\eeas
where 
\beas F_9(a,\rho)=-(1-\rho^2)(1-a\rho)+\rho(1-a\rho)+(k+1)\rho(1+a\rho)^2-\rho (1-a\rho)(1+a\rho)^2.\eeas
It is evident that
\beas \lim_{a\to1^-} F_9(a,\rho)
=-(1-\rho)\left(1-\rho^2-\rho\right)+\rho(1+\rho)^2\left(k+\rho\right)>0\quad\text{for}\quad\rho>\rho_0, \eeas
where $\rho_0$ is the unique positive root of the equation (\ref{p5}) in $\left(0, \left(\sqrt{5}-1\right)/2\right)$. 
This shows that the radius $\rho_0$ is the best possible. This completes the proof.
\end{proof}
Letting  $K\to\infty$ in \textrm{Theorems \ref{TT1}} and \ref{T7}, we obtain the following sharp harmonic analogues of the classical Bohr inequality, respectively.
\begin{cor}  Suppose that $f(z)=h(z)+\ol{g(z)}=\sum_{n=0}^\infty a_n z^n+\ol{\sum_{n=1}^\infty b_n z^n}$ is a sense-preserving harmonic mapping in $\mathbb{D}$, where $\Vert h(z)\Vert_\infty\leq 1$ in $\mathbb{D}$. Then
\beas |h(z)|+\left|h'(z)\right|\rho+\sum_{n=2}^\infty \left|a_n\right|\rho^n+\sum_{n=1}^\infty |b_n|\rho^n\leq 1\quad\text{for}\quad \rho\leq R_1\leq \sqrt{2}-1,\eeas
where $R_1(\approx 0.1671)$ is the unique positive root of the equation 
\beas 2\rho^4+ 5\rho^3+ 5\rho^2+5\rho-1=0.\eeas 
The number $R_1$ is the best possible.
\end{cor}
\begin{cor}  Suppose that $f(z)=h(z)+\ol{g(z)}=\sum_{n=0}^\infty a_n z^n+\ol{\sum_{n=1}^\infty b_n z^n}$ is a sense-preserving harmonic mapping in $\mathbb{D}$, where $\Vert h(z)\Vert_\infty\leq 1$ in $\mathbb{D}$. Then
\beas |h(z)|^2+\left|h'(z)\right|\rho+\sum_{n=2}^\infty \left|a_n\right|\rho^n+\sum_{n=1}^\infty |b_n|\rho^n\leq 1\quad\text{for}\quad \rho\leq R_1\leq (\sqrt{5}-1)/2,\eeas
where $R_1(\approx 0.255508)$ is the unique positive root of the equation 
\beas \rho^4+2\rho^3+3\rho^2+3\rho-1=0.\eeas 
The number $R_1$ is the best possible.
\end{cor}
The following result is the sharp Bohr-Rogosinski inequality in the settings of \textrm{Theorem G}. 
\begin{theo}\label{T2} Suppose that $f(z)=h(z)+\ol{g(z)}=\sum_{n=0}^\infty a_n z^n+\ol{\sum_{n=1}^\infty b_n z^n}$ is a sense-preserving $K$-quasiconformal harmonic mapping in $\mathbb{D}$, where $h(z)$ satisfies the conditions $\text{Re}(h(z))<1$ in $\mathbb{D}$ and $h(0)=a_0>0$. Then, for any $p\in\N$, we have 
\beas a_0+|h(z)-a_0|^p+\sum_{n=1}^\infty |a_n|\rho^n+\sum_{n=1}^\infty |b_n|\rho^n\leq 1\quad\text{for}\quad \rho\leq \rho_0,\eeas
where $\rho_0\in(0,1)$ is the unique root of the equation
\beas \left(\frac{2\rho}{1-\rho}\right)^p+\frac{4K\rho}{(K+1)(1-\rho)}-1=0.\eeas
The number $\rho_0$ is sharp.\end{theo}
\begin{proof} Let $p(z)=1-h(z)$ for $z\in\D$. Then, $\text{Re}\;p(z)>0$ in $\D$. In view of \textrm{Lemma \ref{lem6}}, we have $|a_n|\leq 2(1-a_0)$ for $n\geq 1$.
Using similar arguments as in the proof of \textrm{Theorem \ref{T1}} and in view of \textrm{Lemma \ref{lem3}} together with the condition $|g'(z)|\leq k|h'(z)|$, we have
\beas&& \sum_{n=1}^\infty |b_n|^2\rho^n\leq k^2 \sum_{n=1}^\infty |a_n|^2\rho^n\leq 4k^2(1-a_0)^2\frac{\rho}{1-\rho},\\[2mm]
&&\sum_{n=1}^\infty |b_n|\rho^n\leq \left(\sum_{n=1}^\infty |b_n|^2\rho^n\right)^{1/2}\left(\sum_{n=1}^\infty\rho^n\right)^{1/2} \leq 2k(1-a_0)\frac{\rho}{1-\rho}
\\[2mm]\text{and}&&|h(z)-a_0|=\left|\sum_{n=1}^\infty a_n z^n\right|\leq \sum_{n=1}^\infty |a_n| \rho^{n}\leq 2(1-a_0)\frac{\rho}{1-\rho}.
\eeas
Therefore,
\beas a_0+|h(z)-a_0|^p+\sum_{n=1}^\infty |a_n|\rho^n+\sum_{n=1}^\infty |b_n|\rho^n&\leq& a_0+2^p (1-a_0)^p\frac{\rho^p}{(1-\rho)^p}\\
&&+2(k+1)(1-a_0)\frac{\rho}{1-\rho}\\
&=&1+(1-a_0)F_{10}(a_0,\rho),\eeas
where 
\beas F_{10}(a_0,\rho)=2^p (1-a_0)^{p-1}\frac{\rho^p}{(1-\rho)^p}+2(k+1)\frac{\rho}{1-\rho}-1\eeas
Differentiating $F_{10}(a_0,\rho)$ partially with respect to $a_0$, we get
\beas \frac{\pa}{\pa a_0} F_{10}(a_0,\rho)=-(p-1)(1-a_0)^{p-2}\left(\frac{2\rho}{1-\rho}\right)^p\leq 0.\eeas 
Therefore, $F_{10}(a_0,\rho)$ is a monotonically decreasing function of $a_0\in[0,1)$ and it follows that 
\beas F_{10}(a_0,\rho)\leq F_{10}(0,\rho)=\left(\frac{2\rho}{1-\rho}\right)^p+(k+1)\frac{2\rho}{1-\rho}-1\leq 0\quad\text{for}\quad \rho\leq \rho_0,\eeas
where $\rho_0\in(0,1)$ is the smallest root of the equation
\bea\label{p1} F_{11}(\rho):=\left(\frac{2\rho}{1-\rho}\right)^p+\frac{4K\rho}{(K+1)(1-\rho)}-1=0,\eea
where $k=(K-1)/(K+1)$. It is easy to see that $F_{11}(0)=-1<0$, $\lim_{\rho\to1^-}F_{11}(\rho)=+\infty$ and
\beas F_{11}'(\rho)=2^p p \left(\frac{\rho}{1-\rho}\right)^{p-1}\frac{1}{(1-\rho)^2}+\frac{4K}{(K+1)(1-\rho)^2}\geq 0\quad\text{for}\quad \rho\in[0,1),\eeas
which shows that $F_{11}(\rho)$ is a monotonically increasing function of $\rho$. Therefore, $\rho_0\in(0,1)$ is the unique root of the equation (\ref{p1}).\\[2mm]
\indent To prove the sharpness of the result, we consider the function $f_4(z)=h_4(z)+\ol{g_4(z)}$ in $\mathbb{D}$ such that 
\beas h_4(z)=a-2(1-a)\frac{z}{1+z}=A_0+\sum_{n=1}^\infty A_n z^n,\eeas
where $A_0=a\in(0,1)$, $A_n=2(1-a)(-1)^n$ for $n\geq 1$, $a\in[0,1)$ and $g_4(z)=k \sum_{n=1}^\infty A_n z^n$, where $k=(K-1)/(K+1)$. Therefore, we have
\beas &&A_0+|h_4(-\rho)-A_0|^p+\sum_{n=1}^\infty |A_n|\rho^n+\sum_{n=1}^\infty |kA_n|\rho^n\\
&&=a+\left(2(1-a)\frac{\rho}{1-\rho}\right)^p+2(1+k)(1-a)\frac{\rho}{1-\rho}\\[2mm]
&&=1+(1-a)F_{12}(a,\rho),\eeas
where 
\beas F_{12}(a,\rho)=(1-a)^{p-1}\left(\frac{2\rho}{1-\rho}\right)^p+2(1+k)\frac{\rho}{1-\rho}-1.\eeas 
It is evident that
\beas \lim_{a\to0^+}F_{12}(a,\rho)=\left(\frac{2\rho}{1-\rho}\right)^p+2(1+k)\frac{\rho}{1-\rho}-1>0\quad\text{for}\quad \rho>\rho_0,\eeas
where $\rho_0$ is the unique root of the equation (\ref{p1}) in $(0,1)$. This shows that $\rho_0$ is best possible. This completes the proof.\end{proof}
\begin{theo} Suppose that $f(z)=h(z)+\ol{g(z)}=\sum_{n=0}^\infty a_n z^n+\ol{\sum_{n=1}^\infty b_n z^n}$ is a sense-preserving $K$-quasiconformal harmonic mapping in $\mathbb{D}$, where $h(z)$ satisfies the conditions $\text{Re}(h(z))<1$ in $\mathbb{D}$ and $h(0)=a_0>0$. Then, for any $p\in\N$, we have 
\beas a_0^2+|h(z)-a_0|^p+\sum_{n=1}^\infty |a_n|\rho^n+\sum_{n=1}^\infty |b_n|\rho^n\leq 1\quad\text{for}\quad \rho\leq \rho_0,\eeas
where $\rho_0\in(0,1)$ is the unique root of the equation
\beas \left(\frac{2\rho}{1-\rho}\right)^p+\frac{4K\rho}{(K+1)(1-\rho)}-1=0.\eeas
The radius $\rho_0$ is sharp.\end{theo}
\begin{proof} 
By employing analogous reasoning to that utilized in the proof of \textrm{Theorem \ref{T2}}, we arrive at the desired conclusion.\end{proof}
In the following, we obtain the sharp refined version of Bohr-type inequality in the settings of \textrm{Theorem E} without compromising the radius. 
\begin{theo}\label{T3} Suppose that $f(z)=h(z)+\ol{g(z)}=\sum_{n=0}^\infty a_n z^n+\ol{\sum_{n=1}^\infty b_n z^n}$ is a sense-preserving $K$-quasiconformal harmonic mapping in $\mathbb{D}$, where $h(z)$ is bounded in $\mathbb{D}$. Then
\beas&&\sum_{n=0}^\infty |a_n|\rho^n+\left(\frac{1}{1+|a_0|}+\frac{\rho}{1-\rho}\right)\sum_{n=1}^\infty |a_n|^2\rho^{2n}+\sum_{n=1}^\infty |b_n|\rho^n\\
&&+\frac{8K^2(3K+1)^2}{(5K+1)^2(K+1)^2} \frac{S_\rho(h)}{\pi}\leq \Vert h(z)\Vert_\infty\quad\text{for}\quad \rho\leq \rho_0=\frac{K+1}{5K+1}.\eeas
The numbers $8K^2(3K+1)^2/((5K+1)^2(K+1)^2)$ and  $(K+1)/(5K+1)$ cannot be replaced by larger values.
\end{theo}
\begin{proof}
For simplicity, we assume that $\Vert h(z)\Vert_\infty\leq 1$. Then, we have $|a_n|\leq 1-|a_0|^2$ for $n\geq 1$. Using similar arguments as in the proof of \textrm{Theorem \ref{T1}}, and considering \textrm{Lemmas \ref{lem1}, \ref{lem2}}, and \ref{lem3} together with the condition $|g'(z)|\leq k|h'(z)|$, we have
\beas\sum_{n=1}^\infty |b_n|\rho^n\leq \left(\sum_{n=1}^\infty |b_n|^2\rho^n\right)^{1/2}\left(\sum_{n=1}^\infty\rho^n\right)^{1/2} \leq k(1-a^2)\frac{\rho}{1-\rho},\eeas
where $|a_0|=a\in[0,1)$.
From (\ref{E2}) and by \textrm{Lemma \ref{lem5}}, we have 
\beas &&\sum_{n=0}^\infty |a_n|\rho^n+\left(\frac{1}{1+|a_0|}+\frac{\rho}{1-\rho}\right)\sum_{n=1}^\infty |a_n|^2\rho^{2n}+\sum_{n=1}^\infty |b_n|\rho^n+\lambda \frac{S_\rho(h)}{\pi}\\[2mm]
&&\leq a+(1-a^2)\frac{\rho}{1-\rho}+k(1-a^2)\frac{\rho}{1-\rho}+\lambda (1-a^2)^2\frac{\rho^2}{(1-\rho^2)^2}.
\eeas
The remaining calculations are derived from \textrm{Theorem \ref{T1}}.\\[2mm]
\indent To prove the sharpness of the result, we consider the function $f_6(z)=h_6(z)+\ol{g_6(z)}$ in $\mathbb{D}$ such that 
\beas h_6(z)=\frac{a-z}{1-az}=A_0+\sum_{n=1}^\infty A_n z^n,\eeas
where $A_0=a$, $A_n=-(1-a^2)a^{n-1}$ for $n\geq 1$, $a\in[0,1)$ and $g_6(z)=\lambda k \sum_{n=1}^\infty A_n z^n$, where $|\lambda|=1$ and $k=(K-1)/(K+1)$.
Thus,
\beas S_2:&=&\sum_{n=0}^\infty |A_n|\rho^n+\left(\frac{1}{1+|A_0|}+\frac{\rho}{1-\rho}\right)\sum_{n=1}^\infty |A_n|^2\rho^{2n}+\sum_{n=1}^\infty |k\lambda A_n|\rho^n\\[2mm]
&&+\frac{8K^2(3K+1)^2}{(5K+1)^2(K+1)^2}\frac{S_\rho(h_6)}{\pi}\\[2mm]
&=&a+(1+k)(1-a^2)\rho\sum_{n=1}^\infty (a\rho)^{n-1}+\frac{1+a\rho}{(1+a)(1-\rho)}(1-a^2)^2\rho^2\sum_{n=1}^\infty (a\rho)^{2(n-1)}\\[2mm]
&&+\frac{8K^2(3K+1)^2}{(5K+1)^2(K+1)^2}\sum_{n=1}^\infty n|A_n|^2 \rho^{2n}\\[2mm]
&=&a+\frac{(1+k)(1-a^2)\rho}{1-a\rho}+\frac{1+a\rho}{(1+a)(1-\rho)}\frac{(1-a^2)^2\rho^2}{1-a^2\rho^2}\\[2mm]
&&+\frac{8K^2(3K+1)^2}{(5K+1)^2(K+1)^2}(1-a^2)^2\sum_{n=1}^\infty n a^{2(n-1)} \rho^{2n}\\[2mm]
&=& 1+(1-a) G_1(a,\rho),\eeas
where 
\beas G_1(a,\rho)&=&\frac{2K(1+a)\rho}{(K+1)(1-a\rho)}+\frac{8K^2(3K+1)^2}{(5K+1)^2(K+1)^2}\frac{(1-a^2)(1+a)\rho^2}{(1-a^2\rho^2)^2}\\
&&+\frac{(1-a^2)\rho^2}{(1-a\rho)(1-\rho)}-1.\eeas
Differentiating partially $G_1(a,\rho)$ with respect to $\rho$, we have 
\beas &&\frac{\pa}{\pa \rho}G_1(a,\rho)=(1-a^2)\left(\frac{a r^2}{(1 - r) (1 - a r)^2} + \frac{2 r}{(1 - r) (1 - a r)} + \frac{r^2}{(1 - r)^2 (1 - a r)}\right)\\[2mm]
&&+\frac{2K(1+a)}{(K+1)(1-a\rho)^2}+\frac{16K^2(3K+1)^2}{(5K+1)^2(K+1)^2}\frac{(1-a^2)(1+a)\rho(1+\rho^2)}{(1-a^2\rho^2)^3}>0\eeas
for $\rho\in(0,1)$.
Therefore, $G_1(a,\rho)$ is a strictly increasing function of $\rho\in(0,1)$. Thus, for $\rho>(K+1)/(5K+1)$, we have
\beas G_1(a,\rho)&>& G_1(a,(K+1)/(5K+1))\\[2mm]&=&\frac{2K(1+a)}{4K}+(1-a^2)\frac{(1+K)(1+5K)^2(1-a+9K-aK)} {16 K^2(1-a+5K-a K)^2}\\[2mm]
&&+\frac{8K^2(3K+1)^2(1-a^2)(1+a)}{((5+a)K+1+a)^2((5-a)K+1-a)^2}-1\to 0\;\text{as}\;a\to1^-.\eeas 
 Hence $S_2:=1+(1-a) G_1(a,\rho)>1$ for $\rho>(K+1)/(5K+1)$. 												
This shows that $(K+1)/(5K+1)$ is the best possible. This completes the proof.
\end{proof}
The following result is the sharp refined Bohr-Rogosinski inequality in the settings of \textrm{Theorem E}. 
\begin{theo}\label{T4} Suppose that $f(z)=h(z)+\ol{g(z)}=\sum_{n=0}^\infty a_n z^n+\ol{\sum_{n=1}^\infty b_n z^n}$ is a sense-preserving $K$-quasiconformal harmonic mapping in $\mathbb{D}$, where $\Vert h(z)\Vert_\infty\leq 1$ in $\mathbb{D}$. For $p\in(0,2]$, we have 
\beas |h(z)|^p+\sum_{n=1}^\infty |a_n|\rho^n+\left(\frac{1}{1+|a_0|}+\frac{\rho}{1-\rho}\right)\sum_{n=1}^\infty |a_n|^2\rho^{2n}+\sum_{n=1}^\infty |b_n|\rho^n \leq 1\quad\text{for}\quad \rho\leq \rho_0,\eeas
where $\rho_0$ is the unique positive root of the equation 
\beas p(1-\rho)^2-\frac{4K}{K+1}\rho(1+\rho)=0.\eeas The number $\rho_0$ is sharp.
\end{theo}
\begin{proof}
Using similar arguments as in the proof of \textrm{Theorem \ref{T1}}, and in view of \textrm{Lemmas \ref{lem1}, \ref{lem2}}, and \ref{lem3} together with the condition $|g'(z)|\leq k|h'(z)|$, we have
\beas
&&\sum_{n=1}^\infty |b_n|\rho^n\leq \left(\sum_{n=1}^\infty |b_n|^2\rho^n\right)^{1/2}\left(\sum_{n=1}^\infty\rho^n\right)^{1/2} \leq k(1-a^2)\frac{\rho}{1-\rho}\\[2mm]\text{and}&&|h(z)|\leq \frac{|h(0)|+|z|}{1+|h(0)||z|},\eeas
where $|a_0|=a\in[0,1)$.
In view of \textrm{Lemma \ref{lem5}}, we have 
\beas &&|h(z)|^p+\sum_{n=1}^\infty |a_n|\rho^n+\left(\frac{1}{1+|a_0|}+\frac{\rho}{1-\rho}\right)\sum_{n=1}^\infty |a_n|^2\rho^{2n}+\sum_{n=1}^\infty |b_n|\rho^n \\[2mm]
&&\leq \left(\frac{a+\rho}{1+a\rho}\right)^p+(k+1)(1-a^2)\frac{\rho}{1-\rho}\\[2mm]
&&=1+G_2(a,\rho),\eeas
where \beas G_2(a,\rho)=\left(\frac{a+\rho}{1+a\rho}\right)^p+(k+1)(1-a^2)\frac{\rho}{1-\rho}-1.\eeas  
We now consider the following cases.\\
{\bf Case 1.} Let $p\in(0,1]$.
Differentiating partially $G_2(a,\rho)$ twice with respect to $a$, we have 
\bea\label{a1}\frac{\pa}{\pa a}G_2(a,\rho)&=&\frac{p(a+\rho)^{p-1}(1-\rho^2)}{(1+a\rho)^{p+1}}-2a(k+1)\frac{\rho}{1-\rho}\\[2mm]
\frac{\pa^2}{\pa a^2}G_2(a,\rho)&=&\frac{p(1-\rho^2)(a+\rho)^{p-2}}{(1+a\rho)^{p+2}}\left((p-1)(1+a\rho)-(p+1)(a+\rho)\rho\right)\nonumber\\[2mm]
&&-2(k+1)\frac{\rho}{1-\rho}\leq 0\nonumber\eea
for $a\in[0,1)$ and $p\in(0,1]$. Therefore, $\frac{\pa}{\pa a}G_2(a,\rho)$ is a monotonically decreasing function of $a\in[0,1)$. Thus, we have 
\beas \frac{\pa}{\pa a}G_2(a,\rho)\geq \frac{\pa}{\pa a}G_2(1,\rho)
=\frac{p(1-\rho)^2-2(k+1)\rho(1+\rho)}{1-\rho^2}\geq 0,\eeas
for $\rho\leq \rho_0$, where $\rho_0$ is the unique positive root of the equation $p(1-\rho)^2-2(k+1)\rho(1+\rho)=0$. 
Therefore, $G_2(a,\rho)$ is a monotonically increasing function of $a\in[0,1)$ and it follows that 
\beas G_2(a,\rho)\leq G_2(1,\rho)=0\quad\text{for}\quad \rho\leq \rho_0.\eeas 
{\bf Case 2.} Let $p\in(1,2]$. From (\ref{a1}), we have 
\bea\label{a2} \frac{\pa}{\pa a}G_2(a,\rho)&=&\frac{p(a+\rho)^{p-1}(1-\rho^2)}{(1+a\rho)^{p+1}}-2a(k+1)\frac{\rho}{1-\rho}\nonumber\\[2mm]
&=&\frac{p(1-\rho)}{(1+\rho)}G_3(a,r)-2a(k+1)\frac{\rho}{1-\rho},\eea
where 
\beas G_3(a,\rho)=\frac{(1+\rho)^2(a+\rho)^{p-1}}{(1+a\rho)^{p+1}}.\eeas
Differentiating partially $G_3(a,\rho)$ with respect to $\rho$, we see that
\beas\frac{\pa}{\pa a}G_3(a,\rho)
=\frac{(1-a)(1+r)(a+r)^{p-2}}{(1 + a r)^{p+2}}\left(r(p(a+1)+1-a)+a(1+p)+p-1)\right)\geq 0\eeas
 for $\rho\in[0,1)$.
Thus, $G_3(a,\rho)$ is a monotonically increasing function of $\rho\in[0,1)$ and it follows that  
\beas G_3(a,\rho)\geq G_3(a,0)=a^{p-1}\quad\text{for}\quad a\in[0,1). \eeas
From (\ref{a2}), we have 
\beas \frac{\pa}{\pa a}G_2(a,\rho)&\geq& \frac{p(1-\rho)}{(1+\rho)}a^{p-1}-2a(k+1)\frac{\rho}{1-\rho}\\[2mm]
&=&a^{p-1}\left(\frac{p(1-\rho)}{1+\rho}-\frac{2a^{2-p}(1+k)(1+\rho)}{1-\rho}\right)\\[2mm]
&\geq &a^{p-1}\left(\frac{p(1-\rho)}{1+\rho}-\frac{2(1+k)(1+\rho)}{1-\rho}\right)\\[2mm]
&=&a^{p-1}\frac{p(1-\rho)^2-2\rho(1+k)(1+\rho)}{1-\rho^2}\geq 0\eeas
for $\rho\leq \rho_0$, where $\rho_0$ is the unique positive root of the equation $p(1-\rho)^2-2(k+1)\rho(1+\rho)=0$. 
Therefore, $G_2(a,\rho)$ is a monotonically increasing function in $a\in [0, 1)$ and it follows that $G_2(a,\rho)\leq G_2(1,\rho) = 0$, which is true for $\rho\leq \rho_0$.
\\[2mm]
\indent To prove the sharpness of the result, we consider the function $f_7(z)=h_7(z)+\ol{g_7(z)}$ in $\mathbb{D}$ such that 
\beas h_7(z)=\frac{a-z}{1-az}=A_0+\sum_{n=1}^\infty A_n z^n,\eeas
where $A_0=a$, $A_n=-(1-a^2)a^{n-1}$ for $n\geq 1$, $a\in[0,1)$ and $g_7(z)=\lambda k \sum_{n=1}^\infty A_n z^n$, where $|\lambda|=1$ and $k=(K-1)/(K+1)$.
Thus,
\beas &&|h_7(\rho)|^p+\sum_{n=1}^\infty |A_n|\rho^n+\left(\frac{1}{1+|A_0|}+\frac{\rho}{1-\rho}\right)\sum_{n=1}^\infty |A_n|^2\rho^{2n}+\sum_{n=1}^\infty |k\lambda A_n|\rho^n\\
&&=\left(\frac{a+\rho}{1+a\rho}\right)^p+(1+k)(1-a^2)\rho\sum_{n=1}^\infty (a\rho)^{n-1}+\frac{(1+a\rho)(1-a^2)^2\rho^2}{(1+a)(1-\rho)}\sum_{n=1}^\infty (a\rho)^{2(n-1)}\\
&&=\left(\frac{a+\rho}{1+a\rho}\right)^p+\frac{(1+k)(1-a^2)\rho}{1-a\rho}+\frac{(1+a\rho)}{(1+a)(1-\rho)}\frac{(1-a^2)^2\rho^2}{1-a^2\rho^2}\\[2mm]
&&=1+(1-a)G_4(a,\rho),\eeas
where 
\beas G_4(a,\rho)=\frac{1}{(1-a)}\left(\left(\frac{a+\rho}{1+a\rho}\right)^p-1\right)+\frac{(1+k)(1+a)\rho}{1-a\rho}+\frac{(1-a^2)\rho^2}{(1-\rho)(1+a\rho)}.\eeas
For $\rho>\rho_0$, we see that
\beas \lim_{a\to1^-} G_4(a,\rho)&=&\lim_{a\to 1^-}\left(-p\left(\frac{a+\rho}{1+a\rho}\right)^{p-1}\left(\frac{1}{1+a\rho}-\frac{\rho(a+\rho)}{(1+a\rho)^2}\right)\right)+\frac{2(1+k)\rho}{1-\rho}\\[2mm]
&=&-p\left(\frac{1-\rho}{1+\rho}\right)+\frac{2(1+k)\rho}{1-\rho}>0,\eeas
which shows that the radius $\rho_0$ is the best possible. This completes the proof.
\end{proof}
In the following, we obtain the sharp refined version of the Bohr-type inequality in the settings of \textrm{Theorem F} in which $|a_0|$ and $|a_1|$ are replaced by $|h(z)|$ and $|h'(z)|$ in the majorant series, respectively.
\begin{theo}\label{T6} Suppose that $f(z)=h(z)+\ol{g(z)}=\sum_{n=0}^\infty a_n z^n+\ol{\sum_{n=2}^\infty b_n z^n}$ is a sense-preserving $K$-quasiconformal harmonic mapping in $\mathbb{D}$, where $\Vert h(z)\Vert_\infty\leq 1$ in $\mathbb{D}$. Then
\beas |h(z)|+|h'(z)|\rho+\sum_{n=2}^\infty |a_n|\rho^n+\left(\frac{1}{1+|a_0|}+\frac{\rho}{1-\rho}\right)\sum_{n=1}^\infty |a_n|^2\rho^{2n}+\sum_{n=2}^\infty |b_n|\rho^n \leq 1\eeas
for $\rho\leq \rho_0\leq 1/3$, where 
$\rho_0\in(0,1/3)$ is the unique root of the equation 
\beas 2\rho^4+3\rho^3+\rho^2+3\rho+2\frac{(K-1)}{(K+1)}(1+\rho)^2(\rho+(1-\rho)\log(1-\rho))=1.\eeas
The number $\rho_0$ is sharp.
\end{theo}
\begin{proof}
Since $\Vert h(z)\Vert_\infty\leq 1$, in view of \textrm{lemma \ref{lem2}}, we have $|a_n|\leq 1-|a_0|^2$ for $n\geq 1$.
Since $f$ is locally univalent and $K$-quasiconformal sense-preserving harmonic mapping on $\mathbb{D}$ with $g'(0)=b_1=0$, Schwarz's lemma gives that the dilatation 
$\omega=g'/h'$ is analytic in $\mathbb{D}$ and $|\omega(z)|\leq k|z|$, {\it i.e.}, $|g'(z)|\leq k|zh'(z)|$ in $\mathbb{D}$, where $K = (1+k)/(1-k) \geq 1$, $k\in[0,1)$.  
In view of \textrm{Lemma \ref{lem30}}, we have
\bea\label{p2} \sum_{n=2}^\infty n |b_n|\rho^{n-1}\leq k \sum_{n=1}^\infty n|a_n|\rho^n\leq k(1-a^2)\sum_{n=1}^\infty n\rho^n=k(1-a^2)\frac{\rho}{(1-\rho)^2}\eea
for $|z|=\rho\leq 1/3$. Integrate (\ref{p2}) on $[0, \rho]$, we have 
\be\label{p6}\sum_{n=2}^\infty |b_n|\rho^{n}\leq k(1-a^2)\int_{0}^\rho\frac{x}{(1-x)^2} dx=k(1-a^2)\left(\frac{\rho}{1-\rho}+\log(1-\rho)\right)\;\;\text{for}\;\;\rho\leq \frac{1}{3}.\ee
Let $|a_0|=a\in[0,1)$. In view of \textrm{Lemma \ref{lem2}}, we have 
\beas|h(z)|\leq \frac{|h(0)|+|z|}{1+|h(0)||z|}\quad\text{and}\quad \left|h'(z)\right|\leq \frac{1-|h(z)|^2}{1-|z|^2}.\eeas
It is evident that $\rho/(1-\rho^2)\leq 1/2$ for any $\rho\in[0,\sqrt{2}-1]$.
Using similar arguments as in the proof of \textrm{Theorem \ref{TT1}}, and in view of \textrm{Lemma \ref{lem5}}, we have 
\beas &&|h(z)|+|h'(z)|\rho+\sum_{n=2}^\infty |a_n|\rho^n+\left(\frac{1}{1+|a_0|}+\frac{\rho}{1-\rho}\right)\sum_{n=1}^\infty |a_n|^2\rho^{2n}+\sum_{n=2}^\infty |b_n|\rho^n\\[2mm]
&&\leq \frac{a+\rho}{1+a\rho}+\frac{\rho}{1-\rho^2}\left(1-\left(\frac{a+\rho}{1+a\rho}\right)^2\right)+(1-a^2)\frac{\rho^2}{1-\rho}\\
&&+ k(1-a^2)\left(\frac{\rho}{1-\rho}+\log(1-\rho)\right)\\[2mm]
&&=1+\frac{(1-a) G_5(a,\rho)}{(1+a\rho)^2(1-\rho)},\eeas
where 
\beas G_5(a,\rho)=(1+a)H_1(\rho)+(1+a)(1+a\rho)^2(H_2(\rho)+ H_3(\rho))-H_4(\rho)(1+a\rho)\eeas
with $H_1(\rho)=\rho(1-\rho)\geq 0$, $H_2(\rho)=\rho^2\geq 0$, $H_3(\rho)=k\left(\rho+(1-\rho)\log(1-\rho)\right)\geq 0$ and $H_4(\rho)=(1-\rho)^2\geq 0$, and the first inequality hold for any $\rho\leq 1/3\leq\sqrt{2}-1$.
Differentiating partially $G_5(a,\rho)$ twice with respect to $a$, we have 
\beas\frac{\pa}{\pa a}G_5(a,\rho)&=&H_1(\rho)+((1+a\rho)^2+2\rho(1+a)(1+a\rho))(H_2(\rho)+ H_3(\rho))-\rho H_4(\rho),\\[2mm]
\frac{\pa^2}{\pa a^2}G_5(a,\rho)&=&(2(1+a\rho)\rho+2\rho(1+a\rho)+2\rho^2(1+a))(H_2(\rho)+ H_3(\rho))\geq 0.\eeas
Therefore, $\frac{\pa}{\pa a}G_5(a,\rho)$ is a monotonically increasing function of $a\in[0,1)$ and it follows that   
\beas \frac{\pa}{\pa a}G_5(a,\rho)\geq \frac{\pa}{\pa a}G_5(0,\rho)
=\rho^3+2\rho^2+(1+2\rho)H_3(\rho)\geq 0.\eeas 
Therefore, $G_5(a,\rho)$ is a monotonically increasing function of $a\in[0,1)$ and hence, we have 
\beas G_5(a,\rho)\leq G_5(1,\rho)
=2\rho^4+3\rho^3+\rho^2+3\rho-1+2k(1+\rho)^2(\rho+(1-\rho)\log(1-\rho))\leq 0\eeas 
for $\rho\leq \rho_0\leq 1/3$,
where $\rho_0$ is the smallest root of the equation 
\bea\label{p3} 2\rho^4+3\rho^3+\rho^2+3\rho-1+2k(1+\rho)^2(\rho+(1-\rho)\log(1-\rho))=0,\eea
where $k=(K-1)/(K+1)$. Let 
\beas G_6(\rho)=\frac{2\rho^4+3\rho^3+\rho^2+3\rho-1}{(1+\rho)^2(1-\rho)}+2k\left(\frac{\rho}{1-\rho}+\log(1-\rho)\right).\eeas
It is evident that $G_6(0)=-1$, $G_6(1/3)=5/24+2k(1/2 -\log(3/2))>0$ and
\beas G_6'(r)=\frac{4-4\rho+ 14\rho^2+6\rho^3-2\rho^4-2\rho^5+2k (\rho+3\rho^2+3\rho^3+\rho^4)}{(1-\rho)^2(1+\rho)^3}\geq0\eeas
for $\rho\in[0, 1/3]$, which shows that $G_6(\rho)$ is a monotonically increasing function of $\rho$. Therefore, $\rho_0\in(0,1/3)$ is the unique root of the equation (\ref{p3}).\\[2mm]\indent To prove the sharpness of the result, we consider the function $f_8(z)=h_8(z)+\ol{g_8(z)}$ in $\mathbb{D}$ such that 
\beas h_8(z)=\frac{a-z}{1-az}=A_0+\sum_{n=1}^\infty A_n z^n,\eeas
where $A_0=a$, $A_n=-(1-a^2)a^{n-1}$ for $n\geq 1$, $a\in[0,1)$ and $g_8'(z)=\lambda k zh_8'(z)$, where $|\lambda|=1$ and $k=(K-1)/(K+1)$. If $g_8(z)=\sum_{n=2}^\infty B_n z^n$, then 
\beas B_n=-k\lambda \left(\frac{n-1}{n}\right)(1-a^2)a^{n-2}\quad\text{ for}\quad n\geq 2.\eeas
Therefore,
\beas &&|h_8(-\rho)|+|h_8'(-\rho)|\rho+\sum_{n=2}^\infty |A_n|\rho^n+\left(\frac{1}{1+|A_0|}+\frac{\rho}{1-\rho}\right)\sum_{n=1}^\infty |A_n|^2\rho^{2n}+\sum_{n=2}^\infty |B_n|\rho^n\\[2mm]
&&=\frac{a+\rho}{1+a\rho}+\frac{(1-a^2)\rho}{(1+a\rho)^2}+(1-a^2)\rho\sum_{n=2}^\infty (a\rho)^{n-1}+(1-a^2)k\rho^2\sum_{n=2}^\infty \frac{n-1}{n}(a\rho)^{n-2}\\[2mm]
&&+\frac{1+a\rho}{(1+a)(1-\rho)}(1-a^2)^2\rho^2\sum_{n=1}^\infty (a\rho)^{2(n-1)}\\[2mm]
&&=\frac{a+\rho}{1+a\rho}+\frac{(1-a^2)\rho}{(1+a\rho)^2}+\frac{(1-a^2)a\rho^2}{1-a\rho}+k(1-a^2)\frac{a\rho+(1-a\rho)\log(1-a\rho)}{a^2(1-a\rho)}\\[2mm]
&&+\frac{1+a\rho}{(1+a)(1-\rho)}\frac{(1-a^2)^2\rho^2}{1-a^2\rho^2}\\[2mm]
&&=1+(1-a)G_7(a,\rho),\eeas
where 
\beas&& G_7(a,\rho)=\frac{(1+a)\rho}{(1+a\rho)^2}+\frac{(1+a)a\rho^2}{1-a\rho}+k(1+a)\frac{a\rho+(1-a\rho)\log(1-a\rho)}{a^2(1-a\rho)}-\frac{1-\rho}{1+a\rho}\\
&&+\frac{(1-a^2)\rho^2}{(1-a\rho)(1-\rho)}.\eeas
It is evident that
\beas \lim_{a\to1^-} G_7(a,\rho)
=\frac{2\rho^4+3\rho^3+\rho^2+3\rho-1+2k(1+\rho)^2(\rho+(1-\rho)\log(1-\rho))}{(1+\rho)^2(1-\rho)}>0\eeas
for $\rho>\rho_0$, where $k=(K-1)/(K+1)$ and $\rho_0\in(0,1/3)$ is the unique positive root of the equation (\ref{p3}). This shows that the number $\rho_0$ is the best possible.
This completes the proof.
\end{proof}
Letting $K\to \infty$ in \textrm{Theorem \ref{T6}}, then we get the result.
\begin{cor} Suppose that $f(z)=h(z)+\ol{g(z)}=\sum_{n=0}^\infty a_n z^n+\ol{\sum_{n=2}^\infty b_n z^n}$ is a sense-preserving harmonic mapping in $\mathbb{D}$, where $\Vert h(z)\Vert_\infty\leq 1$ in $\mathbb{D}$. Then,
\beas |h(z)|+|h'(z)|\rho+\sum_{n=2}^\infty |a_n|\rho^n+\left(\frac{1}{1+|a_0|}+\frac{\rho}{1-\rho}\right)\sum_{n=1}^\infty |a_n|^2\rho^{2n}+\sum_{n=2}^\infty |b_n|\rho^n \leq 1\eeas
for $\rho\leq \rho_0\leq 1/3$, where $\rho_0=0.254876...$ is the unique positive root of the equation 
\beas G_8(\rho):=2\rho^4+3\rho^3+\rho^2+3\rho-1+2(1+\rho)^2(\rho+(1-\rho)\log(1-\rho))=0,\eeas\end{cor}
as illustrated in Figure \ref{fig1}. The number $\rho_0$ is the best possible.
\begin{figure}[H]
\centering
\includegraphics[scale=0.5]{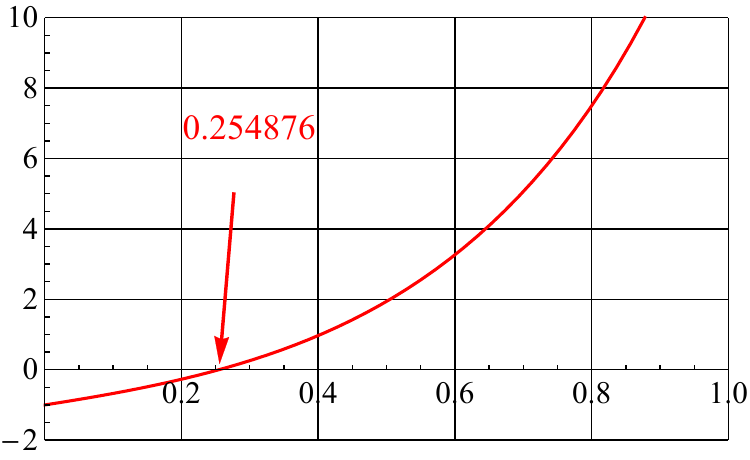}
\caption{The graph of $G_8(\rho)$}
\label{fig1}
\end{figure}
In the following, we obtain the sharp version of \textrm{Theorem F}.
 \begin{theo}Suppose that $f(z)=h(z)+\ol{g(z)}=\sum_{n=0}^\infty a_n z^n+\ol{\sum_{n=2}^\infty b_n z^n}$ is a sense-preserving $K$-quasiconformal harmonic mapping in $\mathbb{D}$, where $h(z)$ is bounded in $\mathbb{D}$. Then
\beas\sum_{n=0}^\infty |a_n|\rho^n+\sum_{n=2}^\infty |b_n|\rho^n\leq \Vert h(z)\Vert_\infty\quad\text{for}\quad \rho\leq \rho_0,\eeas
where $\rho_0$ is the unique positive root of the equation 
\beas \frac{4\rho}{1-\rho}\left(\frac{K}{K+1}\right)+2\left(\frac{K-1}{K+1}\right)\log(1-\rho)=1.\eeas
The number $\rho_0$ is the best possible.
\end{theo}
\begin{proof}Using similar arguments as in the proof of \textrm{Theorem \ref{T1}} and in view of the inequality (\ref{p6}), we have
\beas \sum_{n=0}^\infty |a_n|\rho^n+\sum_{n=2}^\infty |b_n|\rho^n&\leq &1+(1-a)G_9(a,\rho)\;\;\text{for}\;\;\rho\leq \frac{1}{3},\eeas
where $a=|a_0|\in[0,1)$, $G_9(a,\rho)=(k+1)(1+a)\rho/(1-\rho)+k(1+a)\log(1-\rho)-1$ and $k=(K-1)/(K+1)$.
The remaining calculations and the sharpness of the result follow from \textrm{Theorem \ref{T6}}.
\end{proof}
\section{Declarations}
\noindent{\bf Acknowledgment:} The work of the first author is supported by University Grants Commission (IN) fellowship (No. F. 44 - 1/2018 (SA - III)). The authors like to
thank the anonymous reviewers and and the editing team for their valuable suggestions towards the improvement of the paper.\\[2mm]
{\bf Conflict of Interest:} The authors declare that there are no conflicts of interest regarding the publication of this paper.\\[1mm]
{\bf Availability of data and materials:} Not applicable.

\end{document}